\documentclass[12pt]{article}
\usepackage{amsmath, amsthm, amssymb}
\usepackage{array,color,colortbl}

\theoremstyle{plain}
\newtheorem{theorem}{Theorem}
\newtheorem{lemma}[theorem]{Lemma}

\newtheorem{corollary}[theorem]{Corollary}

\usepackage{doc}
\usepackage{makeidx}

\usepackage{tikz}
\usetikzlibrary{intersections,arrows}

\usepackage{epsf}
\usepackage{epstopdf}
\include{psfig}
\input epsf

\input epsf

\usepackage{graphicx}
\graphicspath{ {images/} }

\renewcommand{\geq}{\geqslant}
\renewcommand{\leq}{\leqslant}

\def\dfrac#1#2{\lower0.15ex\hbox{\large$\frac{#1}{#2}$}}

\title{Estimates of the coverage of parameter space by Latin Hypercube and Orthogonal sampling: connections between Populations of Models and Experimental Designs}
\author{
Diane Donovan\thanks{School of Mathematics and Physics, The University of Queensland,
 Queensland 4072,
Australia. \texttt{dmd@maths.uq.edu.au}}\\ \and
Kevin Burrage,\thanks{Department of Computer Science, University of Oxford, UK; ARC Centre of Excellence for Mathematical and Statistical Frontiers, Queensland University of Technology (QUT), Australia \texttt{kevin.burrage@qut.edu.au}}\\ \and
Pamela  Burrage,\thanks{ARC Centre of Excellence for Mathematical and Statistical Frontiers, Queensland University of Technology (QUT), Australia \texttt{pamela.burrage@qut.edu.au}}\\ \and
Thomas A McCourt,\thanks{Department of Mathematics and Statistics, Plymouth University, Plymouth, UK. \texttt{ thomas.mccourt@plymouth.ac.uk}  }
\and
Harold Bevan Thompson\thanks{School of Mathematics and Physics, The University of Queensland,
 Queensland 4072,
Australia. \texttt{hbt@maths.uq.edu.au}}
\and
Emine \c{S}ule Yaz{\i}c{\i}%\footnote{This work was supported by Scientific and Technological Research Council of Turkey TUBITAK Grant Number:110T692}
\thanks{
Department of Mathematics, Ko\c{c} University, Sar{\i}yer,
34450, \.{I}stanbul, Turkey
\texttt{eyazici@ku.edu.tr}}}

\begin{document}

\maketitle

\begin{abstract}

In this paper we use counting arguments to prove that the expected percentage coverage of a $d$ dimensional parameter space of size $n$ when performing $k$ trials with either  Latin Hypercube sampling or Orthogonal sampling (when $n=p^d$) is the same.  We then extend these results to an experimental design setting by projecting onto a 2 dimensional subspace.  In this case the coverage  is equivalent to the Orthogonal sampling setting when the dimension of the parameter space is two.  These results are confirmed by simulations.  The ideas presented here have particular relevance when attempting to perform uncertainty quantification or when building populations of models.

\end{abstract}

%------------------------------------------------------------------------------
\section{Introduction}
\label{sect:introduction} %\cite{B1,B2,B3,B4}

Latin Hypercube sampling (LHS) and its variants such as Orthogonal sampling are becoming key tools in many areas of mathematical modelling such as uncertainty quantification \cite{CG1} and in building populations of models \cite{A1}.   In both cases a key feature is the design of a parameter space sampling strategy, with the goal to achieve the maximum inference by varying multiple parameters at the same time.

 In the case of uncertainty quantification, a quantity, either a random variable or a random response, is often expressed in some basis expansion (Hermite polynomials, for example) and the coefficients can be estimated using some sampling technique.  Such an approach has been used, for instance, to forecast reservoir-performance in the petroleum industry   \cite{LSZ} and to conduct a buckling analysis of a joined-wing model \cite{CG1}. In such cases uncertainty can stem from deficiency of measured data, but also from physical properties such as the heterogeneity of geological formations or buckling response and aeroelastic complications under the effect of compressive loads.
  In these cases, LHS can offer a low cost and relatively uniform coverage  of the sampling space.

  In the setting of a population of models (POM) \cite{MT}  a mathematical model is calibrated by a set of points,  rather than a single point, in parameter space,  all of which are selected to fit sets of experimental/observational data.
The POM approach  was originally proposed for neuroscience modelling, but has been recently extended to cardiac electrophysiology.  Here we highlight some recent research in this context.   \cite{A1} presents a general approach to exploring in-subject variability in cardiac cells. \cite{B3} is also concerned with ionic mechanisms of variability in cardiac cells. \cite{B4} presents a population of  rabbit ventricular action potential models. \cite{B1} explores inter-subject variability in human atrial action potential models and  \cite{A2} presents a population study of mRNA expression levels in failing and non-failing human hearts.  In  these setting, biomarkers, such as Action Potential Duration and beat-to-beat variability, are extracted from time course profiles and then the models are calibrated against these biomarkers. Upper and lower values of each biomarker as observed in the experimental data are used to guarantee that estimates of variability are within biological ranges for any model to be included in the population. If the data cannot be characterised by a set of biomarkers then time course profiles can be used and a normalised root-mean-square (NRMS) comparison between the data values and the simulation values at a set of time points can be used to calibrate the population.

The POM approach leads to methodologies that are essentially probabilistic in nature and it gives greater weight to the experimental, modelling, simulation feedback paradigm \cite{A3}.
By implementing experiments based on a population of models, as distinct from experiments based on a single model, the  variability in the underlying structure can be captured by allowing changes in the parameter values. This avoids complications arising from  decisions on the use of ``best'' or ``mean'' data, and  the difficulties of identifying such data.

We note that POM have similarities with Approximate Bayesian Computation (ABC) \cite{A5}.  However, in ABC the sampling is usually performed adaptively so as to converge to subregions of parameter space where the calibrated models lie, as distinct from random sampling of the entire space.  In this context \cite{B2} have analysed the sensitivity of cardiac cell models using Gaussian emulators. There are many ways to sample the parameter space, depending on costing constraints and therefore limits of computation.  A parameter sweep will cover the whole parameter space at a certain discrete resolution, while random sampling, Latin Hypercube sampling (LHS) and Orthogonal sampling (OS) will give increasingly improved coverage of parameter space when the number of samples is fixed and independent of the dimension of the space.

LHS was first introduced by  McKay, Beckman and Conover  \cite{MBC}. Suppose that the $d$ dimensional parameter space is divided into $n$ equally sized subdivisions in each dimension then a Latin Hypercube (LH) trial is a set of $n$ random samples with one  from each  subdivision; that is, each sample is the only one in each axis-aligned hyperplane containing it. A collection of Latin Hypercube $d$-trials forms a Latin Hypercube sample (LHS).  A variant of LHS, known as Orthogonal sampling, adds the requirement that for each trial the entire sample space must be sampled evenly at some coarse resolution.

An advantage of LHS is that it stratifies each univariate margin simultaneously,  while Stein \cite{stein} showed that  with LHS there is a form of variance reduction compared with uniform random sampling. Tang \cite{A7} suggested that it may also be important to stratify the bivariate margins.  For instance, an Experimental Design  may involve a large number of variables, but in reality only a relatively small number of these variables are  effective. One way of dealing with this problem has been to project the factors onto a subspace spanned by the effective variables. However this can result in a replication of sample points within the effective subspace. Welch et al. \cite{A8}  suggested   LHS as a method for screening for effective factors, but there is still no guarantee, even in the  case of bivariate margins, that this projection is uniformly distributed. Thus as an alternative, Tang \cite{A7} advocated Orthogonal sampling  and proposed a technique based on the existence of orthogonal arrays. Tang goes on to show that  Orthogonal sampling achieves uniformity on small dimensional margins. Tang's approach is to start with an orthogonal array, see \cite{CD}, and to replace its entries by random permutations to obtain an Orthogonal sample.

Given this background it is clearly important to be able to estimate the expected coverage of parameter space using sampling techniques such as LHS or OS.  Furthermore, it is also important to understand the relationship between Experimental Design and POM in this regard.  For example, it may be desirable to know if  a POM calibrates  for interactions of ``small strength''  by checking for all possible combinations of levels for, say, pairs or triples of variables. This would equate to investigating  the coverage of 2 and 3 dimensional subspaces.

In this setting   the authors \cite{A4} focused on  estimating the expected coverage of a 2 dimensional parameter space for a population of $k$ trials forming a LHS with each trial of size $n$. In particular, an incomplete counting argument was used to predict  the expected coverage  of points in the parameter space after $k$ trials of size $n$. These estimates were compared against numerical results based on a MATLAB implementation of  100 simulations. The results of the simulations led the authors to conjecture that the expected percentage coverage by $k$ trials of a 2 dimensional parameter space  tended to $1-e^{-k/n}.$

In a later paper \cite{BBDT} the authors extended this work and reported on the expected coverage of $d$ dimensional space based on MATLAB implementations of simulations of LHS and OS.
They also tested for uniform coverage of lower dimensional subspaces.  They conjectured that the coverage of  a $t$ dimensional subspace of a $d$ dimensional parameter space of size $n$ when performing $k$ trials of Latin Hypercube sampling takes the form
\begin{align}
P(k,n,d,t)=1-(1-1/n^{t-1})^k\mbox{ or }1-e^{-k/n^{t-1}}\label{conj1}
\end{align}
when $k$ is large,   suggesting that the coverage is independent of $d$. This work allowed the authors to make connections between building Populations of Models and Experimental Designs.  Further the results in \cite{BBDT} indicate that Orthogonal sampling is superior to Latin Hypercube sampling in terms of giving a more uniform coverage of the $t$ dimensional subspace at  the sub-block level when only attempting partial coverage of this subspace.

In this  paper we give formal counting arguments and prove that the conjectures given above are true.  We provide counting arguments and use combinatorial techniques to find the expected coverage of parameter space when taking the union of $k$ trials in the case of LHS and OS.  We extend these arguments in a natural manner to sub-block coverage when projecting onto a 2 dimensional subspace (Experimental Design).  We also give theoretical bounds on the percentage coverage of parameter space for both LHS and  OS (showing that they are equivalent with respect to  the coverage).  We then extend these estimates to the coverage when projecting down onto a 2 dimensional subspace.  Finally  we present some simulation results that support our theoretical results along with  concluding remarks.

\section{Methods}

We begin by reviewing  the well known methods used to generate  Latin Hypercube samples and formalise the definitions for Orthogonal samples.

 A Latin Hypercube trial on $d$ variables each taking $n$ values from the set $[n]=\{1,2,\dots,n\}$  may be  represented as an $n$ by $d$ matrix where each column is an arbitrary permutation of $[n]$, and with each row forming  a $d$-tuple of the Latin Hypercube trial.
 Thus  a {\em Latin Hypercube trial} (a LH $d$-trial)  is a randomly generated subset of $n$ points from a $d$ dimensional parameter space  satisfying the condition that the projections onto each of the 1 dimensional subspaces are permutations. A collection of $k$ Latin Hypercube $d$-trials forms a {\em Latin Hypercube sample} (LHS).

  By way of an example, LHT1 and LHT2 are two matrices which correspond to two Latin Hypercube $3$-trials on the set $\{1,2,\ldots,8\}$. Note that since a LH $d$-trial is a multiset it is invariant under any permutation of the rows.

 \begin{align*}
 \begin{array}{cccccccc}
 \mbox{ LHT}1&&\mbox{ LHT}2&&&\mbox{ LHT}3&& \mbox{ OT}4\\
\left[\begin{array}{ccc}
1&2&1\\
2&3&3\\
3&1&2\\
4&7&8\\
5&8&5\\
6&5&4\\
7&4&6\\
8&6&7\\
\end{array}\right]&&
\left[\begin{array}{ccc}
1&3&2\\
2&4&6\\
3&5&3\\
4&7&8\\
5&1&1\\
6&2&7\\
7&8&4\\
8&6&5\\
\end{array}\right]&&&
\left[\begin{array}{ccc}
(1,1)&(1,2)&(1,1)\\
(1,2)&(1,3)&(1,3)\\
(1,3)&(1,1)&(1,2)\\
(1,4)&(2,3)&(2,4)\\
(2,1)&(2,4)&(2,1)\\
(2,2)&(2,1)&(1,4)\\
(2,3)&(1,4)&(2,2)\\
(2,4)&(2,2)&(2,3)\\
\end{array}\right]&&
\left[\begin{array}{ccc}
(1,1)&(1,3)&(1,2)\\
(1,2)&(1,4)&(2,2)\\
(1,3)&(2,1)&(1,3)\\
(1,4)&(2,3)&(2,4)\\
(2,1)&(1,1)&(1,1)\\
(2,2)&(1,2)&(2,3)\\
(2,3)&(2,4)&(1,4)\\
(2,4)&(2,2)&(2,1)\\
\end{array}\right]
\end{array}
\end{align*}
 Latin Hypercube $3$-trials LHT1 and LHT2 may also be represented diagrammatically as illustrated in Figure \ref{LHS}. Here the value of the third variable is represented by the colour:

 \begin{center}
 \begin{tabular}{|lc|lc|lc|lc|}
\multicolumn{8}{c}{Third Coordinate}\\
\hline
Light Blue &1&
Light Pink &2&
Light Green &3&
Light Red &4\\
\hline
Dark Blue &5&
Dark Pink &6&
Dark Green &7&
Dark Red &8\\
\hline
\end{tabular}\end{center}

\begin{figure}%[htb]
\caption{Projections onto the 2 dimensional subspaces for LHT1 and LHT2} \label{LHS}

\begin{center}
\begin{tabular}{ccc}

\scalebox{0.55}
{
\begin{tikzpicture}[fill=gray!50, scale=1,vertex/.style={circle,inner sep=1.5,color=blue,fill=blue,draw}]

\draw [yshift=6.5cm, color=gray!30] (2,1) -- (4,2) -- (4,4.5);
\draw [yshift=6.5cm, color=gray!30] (2,4.5) -- (2,2) -- (4,1);
\draw [yshift=6.5cm, color=gray!30] (3,2.5) -- (3,5);
\draw [yshift=6.5cm, color=gray!30] (1,1.5) -- (3,2.5) -- (5,1.5);
\draw [yshift=6.5cm, color=gray!30] (1,2.75) -- (3,3.75) -- (5,2.75);
\draw [yshift=6.5cm, color=gray!30] (1,4) -- (3,5) -- (5,4) -- (5,1.5);
\draw [yshift=6.5cm, ultra thick] (1,4) -- (1,1.5) -- (3,0.5) -- (5,1.5);

\node [yshift=6.5cm] at (3,5.5) {LHT1 8 sample points};

\node [yshift=6.5cm] at (3.25,0.25) {0};
\node [yshift=6.5cm] at (4,0.75) {4};
\node [yshift=6.5cm] at (5,1.25) {8};
\node [yshift=6.5cm] at (4.5,0.5) {$P_1$};

\node [yshift=6.5cm] at (2.75,0.25) {0};
\node [yshift=6.5cm] at (2,0.75) {4};
\node [yshift=6.5cm] at (1,1.25) {8};
\node [yshift=6.5cm] at (1.5,0.5) {$P_2$};

\node [yshift=6.5cm] at (0.75,1.5) {0};
\node [yshift=6.5cm] at (0.75,2.75) {4};
\node [yshift=6.5cm] at (0.75,4) {8};
\node [yshift=6.5cm] at (0.25,2.75) {$P_3$};

\node [yshift=6.5cm] at (2.75,1.1875) [vertex]{};
\node [yshift=6.5cm] at (2.75,2.0625) [vertex]{};
\node [yshift=6.5cm] at (3.5,1.625) [vertex]{};
\node [yshift=6.5cm] at (2.25,4.375) [vertex]{};
\node [yshift=6.5cm] at (2.25,3.6875) [vertex]{};
\node [yshift=6.5cm] at (3.25,3.125) [vertex]{};
\node [yshift=6.5cm] at (3.75,3.75) [vertex]{};
\node [yshift=6.5cm] at (3.5,4.4375) [vertex]{};

%%%%%%%%%%%%%%%%%%%%%%%%%%%%%%%%%%%%%%%%%%%%%%%%%%%%%%%%%%%%%%

\xdefinecolor{darkgreen}{RGB}{143,188,143}
\xdefinecolor{darkred}{RGB}{139,0,0}

\node [xshift=6.5cm, yshift=6.5cm] at (3,5.5) {Projection onto {$P_1,P_2$} subspace};

\node [xshift=6.5cm, yshift=6.5cm] at (3,0.25) {$P_1$};
\node [xshift=6.5cm, yshift=6.5cm] at (0.25,3) {$P_2$};

\foreach \i in {1,...,8}
\node [xshift=6.5cm, yshift=6.5cm] at (0.75+0.5*\i,0.75) {\i};

\foreach \j in {1,...,8}
\node [xshift=6.5cm, yshift=6.5cm] at (0.75,0.75+0.5*\j) {\j};

\draw[xshift=6.5cm, yshift=6.5cm, fill=blue!30] (1,1.5) -- (1.5,1.5) -- (1.5,2) -- (1,2) -- cycle;
\node [xshift=6.5cm, yshift=6.5cm] at (1.25,1.75) {$1$};

\draw[xshift=6.5cm, yshift=6.5cm, fill=pink!50] (2,1) -- (2.5,1) -- (2.5,1.5) -- (2,1.5) -- cycle;
\node [xshift=6.5cm, yshift=6.5cm] at (2.25,1.25) {$2$};

\draw[xshift=6.5cm, yshift=6.5cm, fill=green!30] (1.5,2) -- (2,2) -- (2,2.5) -- (1.5,2.5) -- cycle;
\node [xshift=6.5cm, yshift=6.5cm] at (1.75,2.25) {$3$};

\draw[xshift=6.5cm, yshift=6.5cm, fill=red!100] (3.5,3) -- (4,3) -- (4,3.5) -- (3.5,3.5) -- cycle;
\node [xshift=6.5cm, yshift=6.5cm] at (3.75,3.25) {$4$};

\draw[xshift=6.5cm, yshift=6.5cm, fill=blue!80] (3,4.5) -- (3.5,4.5) -- (3.5,5) -- (3,5) -- cycle;
\node [xshift=6.5cm, yshift=6.5cm] at (3.25,4.75) {$5$};

\draw[xshift=6.5cm, yshift=6.5cm, fill=pink!150] (4,2.5) -- (4.5,2.5) -- (4.5,3) -- (4,3) -- cycle;
\node [xshift=6.5cm, yshift=6.5cm] at (4.25,2.75) {$6$};

\draw[xshift=6.5cm, yshift=6.5cm, fill=darkgreen!100] (4.5,3.5) -- (5,3.5) -- (5,4) -- (4.5,4) -- cycle;
\node [xshift=6.5cm, yshift=6.5cm] at (4.75,3.75) {$7$};

\draw[xshift=6.5cm, yshift=6.5cm, fill=darkred!100] (2.5,4) -- (3,4) -- (3,4.5) -- (2.5,4.5) -- cycle;
\node [xshift=6.5cm, yshift=6.5cm] at (2.75,4.25) {$8$};

\foreach \x in {0,...,6}
	\draw[xshift=6.5cm, yshift=6.5cm, color=gray!30] (0.5*\x+1.5,1)--(0.5*\x+1.5,5) ;
\foreach \x in {0,...,6}	
	\draw[xshift=6.5cm, yshift=6.5cm, color=gray!30] (1,0.5*\x+1.5)--(5,0.5*\x+1.5) ;
\draw[xshift=6.5cm, yshift=6.5cm, color=gray!30] (1,1) -- (5,1) -- (5,5) -- (1,5) -- cycle;
\draw[xshift=6.5cm, yshift=6.5cm, ultra thick] (3,1) -- (3,5);
\draw[xshift=6.5cm, yshift=6.5cm, ultra thick] (1,3) -- (5,3);

%%%%%%%%%%%%%%%%%%%%%%%%%%%%%%%%%%%%%%%%%%%%%%%%%%%%%%%%%%%%%%

\node at (3,5.5) {Projection onto {$P_1,P_3$} subspace};

\node at (3,0.25) {$P_1$};
\node at (0.25,3) {$P_3$};

\foreach \i in {1,...,8}
\node at (0.75+0.5*\i,0.75) {\i};

\foreach \j in {1,...,8}
\node at (0.75,0.75+0.5*\j) {\j};

\draw[fill=blue!30] (2,1.5) -- (2.5,1.5) -- (2.5,2) -- (2,2) -- cycle;
\node at (2.25,1.75) {$1$};

\draw[fill=pink!50] (1,1) -- (1.5,1) -- (1.5,1.5) -- (1,1.5) -- cycle;
\node at (1.25,1.25) {$2$};

\draw[fill=green!30] (1.5,2) -- (2,2) -- (2,2.5) -- (1.5,2.5) -- cycle;
\node at (1.75,2.25) {$3$};

\draw[fill=red!100] (4,3.5) -- (4.5,3.5) -- (4.5,4) -- (4,4) -- cycle;
\node at (4.25,3.75) {$4$};

\draw[fill=blue!80] (3.5,2.5) -- (4,2.5) -- (4,3) -- (3.5,3) -- cycle;
\node at (3.75,2.75) {$5$};

\draw[fill=pink!150] (4.5,4) -- (5,4) -- (5,4.5) -- (4.5,4.5) -- cycle;
\node at (4.75,4.25) {$6$};

\draw[fill=darkgreen!100] (2.5,4.5) -- (3,4.5) -- (3,5) -- (2.5,5) -- cycle;
\node at (2.75,4.75) {$7$};

\draw[fill=darkred!100] (3,3) -- (3.5,3) -- (3.5,3.5) -- (3,3.5) -- cycle;
\node at (3.25,3.25) {$8$};

\foreach \x in {0,...,6}
	\draw[color=gray!30] (0.5*\x+1.5,1)--(0.5*\x+1.5,5) ;
\foreach \x in {0,...,6}	
	\draw[color=gray!30] (1,0.5*\x+1.5)--(5,0.5*\x+1.5) ;
\draw[color=gray!30] (1,1) -- (5,1) -- (5,5) -- (1,5) -- cycle;
\draw[ultra thick] (3,1) -- (3,5);
\draw[ultra thick] (1,3) -- (5,3);

%%%%%%%%%%%%%%%%%%%%%%%%%%%%%%%%%%%%%%%%%%%%%%%%%%%%%%%%%%%%%%

\node [xshift=6.5cm] at (3,5.5) {Projection onto {$P_2,P_3$} subspace};

\node [xshift=6.5cm] at (3,0.25) {$P_2$};
\node [xshift=6.5cm] at (0.25,3) {$P_3$};

\foreach \i in {1,...,8}
\node [xshift=6.5cm] at (0.75+0.5*\i,0.75) {\i};

\foreach \j in {1,...,8}
\node [xshift=6.5cm] at (0.75,0.75+0.5*\j) {\j};

\draw[xshift=6.5cm, fill=blue!30] (1.5,1) -- (2,1) -- (2,1.5) -- (1.5,1.5) -- cycle;
\node [xshift=6.5cm] at (1.75,1.25) {$1$};

\draw[xshift=6.5cm, fill=pink!50] (2,2) -- (2.5,2) -- (2.5,2.5) -- (2,2.5) -- cycle;
\node [xshift=6.5cm] at (2.25,2.25) {$2$};

\draw[xshift=6.5cm, fill=green!30] (1,1.5) -- (1.5,1.5) -- (1.5,2) -- (1,2) -- cycle;
\node [xshift=6.5cm] at (1.25,1.75) {$3$};

\draw[xshift=6.5cm, fill=red!100] (4,4.5) -- (4.5,4.5) -- (4.5,5) -- (4,5) -- cycle;
\node [xshift=6.5cm] at (4.25,4.75) {$4$};

\draw[xshift=6.5cm, fill=blue!80] (4.5,3) -- (5,3) -- (5,3.5) -- (4.5,3.5) -- cycle;
\node [xshift=6.5cm] at (4.75,3.25) {$5$};

\draw[xshift=6.5cm, fill=pink!150] (3,2.5) -- (3.5,2.5) -- (3.5,3) -- (3,3) -- cycle;
\node [xshift=6.5cm] at (3.25,2.75) {$6$};

\draw[xshift=6.5cm, fill=darkgreen!100] (2.5,3.5) -- (3,3.5) -- (3,4) -- (2.5,4) -- cycle;
\node [xshift=6.5cm] at (2.75,3.75) {$7$};

\draw[xshift=6.5cm, fill=darkred!100] (3.5,4) -- (4,4) -- (4,4.5) -- (3.5,4.5) -- cycle;
\node [xshift=6.5cm] at (3.75,4.25) {$8$};

\foreach \x in {0,...,6}
	\draw[xshift=6.5cm, color=gray!30] (0.5*\x+1.5,1)--(0.5*\x+1.5,5) ;
\foreach \x in {0,...,6}	
	\draw[xshift=6.5cm, color=gray!30] (1,0.5*\x+1.5)--(5,0.5*\x+1.5) ;
\draw[xshift=6.5cm, color=gray!30] (1,1) -- (5,1) -- (5,5) -- (1,5) -- cycle;
\draw[xshift=6.5cm, ultra thick] (3,1) -- (3,5);
\draw[xshift=6.5cm, ultra thick] (1,3) -- (5,3);

\end{tikzpicture}
}

%%%%%%%%%%%%%%%%%%%%%%%%%%%%%%%%%%%%%%%%%
%%%%%%%%%%%%%%%%%%%%%%%%%%%%%%%%%%%%%%%%%
&\,&
%%%%%%%%%%%%%%%%%%%%%%%%%%%%%%%%%%%%%%%%%
%%%%%%%%%%%%%%%%%%%%%%%%%%%%%%%%%%%%%%%%%

%Now for LHT2
\scalebox{0.55}
{
\begin{tikzpicture}[fill=gray!50, scale=1,vertex/.style={circle,inner sep=1.5,color=blue,fill=blue,draw}]

\draw [yshift=6.5cm, color=gray!30] (2,1) -- (4,2) -- (4,4.5);
\draw [yshift=6.5cm, color=gray!30] (2,4.5) -- (2,2) -- (4,1);
\draw [yshift=6.5cm, color=gray!30] (3,2.5) -- (3,5);
\draw [yshift=6.5cm, color=gray!30] (1,1.5) -- (3,2.5) -- (5,1.5);
\draw [yshift=6.5cm, color=gray!30] (1,2.75) -- (3,3.75) -- (5,2.75);
\draw [yshift=6.5cm, color=gray!30] (1,4) -- (3,5) -- (5,4) -- (5,1.5);
\draw [yshift=6.5cm, ultra thick] (1,4) -- (1,1.5) -- (3,0.5) -- (5,1.5);

\node [yshift=6.5cm] at (3,5.5) {LHT2 8 sample points};

\node [yshift=6.5cm] at (3.25,0.25) {0};
\node [yshift=6.5cm] at (4,0.75) {4};
\node [yshift=6.5cm] at (5,1.25) {8};
\node [yshift=6.5cm] at (4.5,0.5) {$P_1$};

\node [yshift=6.5cm] at (2.75,0.25) {0};
\node [yshift=6.5cm] at (2,0.75) {4};
\node [yshift=6.5cm] at (1,1.25) {8};
\node [yshift=6.5cm] at (1.5,0.5) {$P_2$};

\node [yshift=6.5cm] at (0.75,1.5) {0};
\node [yshift=6.5cm] at (0.75,2.75) {4};
\node [yshift=6.5cm] at (0.75,4) {8};
\node [yshift=6.5cm] at (0.25,2.75) {$P_3$};

\node [yshift=6.5cm] at (2.5,1.625) [vertex]{};
\node [yshift=6.5cm] at (2.5,3.125) [vertex]{};
\node [yshift=6.5cm] at (2.5,2.4375) [vertex]{};
\node [yshift=6.5cm] at (2.25,4.375) [vertex]{};
\node [yshift=6.5cm] at (4,1.5625) [vertex]{};
\node [yshift=6.5cm] at (4,3.6875) [vertex]{};
\node [yshift=6.5cm] at (2.75,3.625) [vertex]{};
\node [yshift=6.5cm] at (3.5,3.8125) [vertex]{};

%%%%%%%%%%%%%%%%%%%%%%%%%%%%%%%%%%%%%%%%%%%%%%%%%%%%%%%%%%%%%%
%%P1,P2
%%%%%%%%%%%%%%%%%%%%%%%%%%%%%%%%%%%%%%%%%%%%%%%%%%%%%%%%%%%%%%

\xdefinecolor{darkgreen}{RGB}{143,188,143}
\xdefinecolor{darkred}{RGB}{139,0,0}

\node [xshift=6.5cm, yshift=6.5cm] at (3,5.5) {Projection onto {$P_1,P_2$} subspace};

\node [xshift=6.5cm, yshift=6.5cm] at (3,0.25) {$P_1$};
\node [xshift=6.5cm, yshift=6.5cm] at (0.25,3) {$P_2$};

\foreach \i in {1,...,8}
\node [xshift=6.5cm, yshift=6.5cm] at (0.75+0.5*\i,0.75) {\i};

\foreach \j in {1,...,8}
\node [xshift=6.5cm, yshift=6.5cm] at (0.75,0.75+0.5*\j) {\j};

\draw[xshift=6.5cm, yshift=6.5cm, fill=blue!30] (3,1) -- (3.5,1) -- (3.5,1.5) -- (3,1.5) -- cycle;
\node [xshift=6.5cm, yshift=6.5cm] at (3.25,1.25) {$1$};

\draw[xshift=6.5cm, yshift=6.5cm, fill=pink!50] (1,2) -- (1.5,2) -- (1.5,2.5) -- (1,2.5) -- cycle;
\node [xshift=6.5cm, yshift=6.5cm] at (1.25,2.25) {$2$};

\draw[xshift=6.5cm, yshift=6.5cm, fill=green!30] (2,3) -- (2.5,3) -- (2.5,3.5) -- (2,3.5) -- cycle;
\node [xshift=6.5cm, yshift=6.5cm] at (2.25,3.25) {$3$};

\draw[xshift=6.5cm, yshift=6.5cm, fill=red!100] (4,4.5) -- (4.5,4.5) -- (4.5,5) -- (4,5) -- cycle;
\node [xshift=6.5cm, yshift=6.5cm] at (4.25,4.75) {$4$};

\draw[xshift=6.5cm, yshift=6.5cm, fill=blue!80] (4.5,3.5) -- (5,3.5) -- (5,4) -- (4.5,4) -- cycle;
\node [xshift=6.5cm, yshift=6.5cm] at (4.75,3.75) {$5$};

\draw[xshift=6.5cm, yshift=6.5cm, fill=pink!150] (1.5,2.5) -- (2,2.5) -- (2,3) -- (1.5,3) -- cycle;
\node [xshift=6.5cm, yshift=6.5cm] at (1.75,2.75) {$6$};

\draw[xshift=6.5cm, yshift=6.5cm, fill=darkgreen!100] (3.5,1.5) -- (4,1.5) -- (4,2) -- (3.5,2) -- cycle;
\node [xshift=6.5cm, yshift=6.5cm] at (3.75,1.75) {$7$};

\draw[xshift=6.5cm, yshift=6.5cm, fill=darkred!100] (2.5,4) -- (3,4) -- (3,4.5) -- (2.5,4.5) -- cycle;
\node [xshift=6.5cm, yshift=6.5cm] at (2.75,4.25) {$8$};

\foreach \x in {0,...,6}
	\draw[xshift=6.5cm, yshift=6.5cm, color=gray!30] (0.5*\x+1.5,1)--(0.5*\x+1.5,5) ;
\foreach \x in {0,...,6}	
	\draw[xshift=6.5cm, yshift=6.5cm, color=gray!30] (1,0.5*\x+1.5)--(5,0.5*\x+1.5) ;
\draw[xshift=6.5cm, yshift=6.5cm, color=gray!30] (1,1) -- (5,1) -- (5,5) -- (1,5) -- cycle;
\draw[xshift=6.5cm, yshift=6.5cm, ultra thick] (3,1) -- (3,5);
\draw[xshift=6.5cm, yshift=6.5cm, ultra thick] (1,3) -- (5,3);

%%%%%%%%%%%%%%%%%%%%%%%%%%%%%%%%%%%%%%%%%%%%%%%%%%%%%%%%%%%%%%
%%P1P3
%%%%%%%%%%%%%%%%%%%%%%%%%%%%%%%%%%%%%%%%%%%%%%%%%%%%%%%%%%%%%%

\node at (3,5.5) {Projection onto {$P_1,P_3$} subspace};

\node at (3,0.25) {$P_1$};
\node at (0.25,3) {$P_3$};

\foreach \i in {1,...,8}
\node at (0.75+0.5*\i,0.75) {\i};

\foreach \j in {1,...,8}
\node at (0.75,0.75+0.5*\j) {\j};

\draw[fill=blue!30] (3,1) -- (3.5,1) -- (3.5,1.5) -- (3,1.5) -- cycle;
\node at (3.25,1.25) {$1$};

\draw[fill=pink!50] (3.5,4) -- (4,4) -- (4,4.5) -- (3.5,4.5) -- cycle;
\node at (3.75,4.25) {$2$};

\draw[fill=green!30] (1,1.5) -- (1.5,1.5) -- (1.5,2) -- (1,2) -- cycle;
\node at (1.25,1.75) {$3$};

\draw[fill=red!100] (1.5,3.5) -- (2,3.5) -- (2,4) -- (1.5,4) -- cycle;
\node at (1.75,3.75) {$4$};

\draw[fill=blue!80] (2,2) -- (2.5,2) -- (2.5,2.5) -- (2,2.5) -- cycle;
\node at (2.25,2.25) {$5$};

\draw[fill=pink!150] (4.5,3) -- (5,3) -- (5,3.5) -- (4.5,3.5) -- cycle;
\node at (4.75,3.25) {$6$};

\draw[fill=darkgreen!100] (2.5,4.5) -- (3,4.5) -- (3,5) -- (2.5,5) -- cycle;
\node at (2.75,4.75) {$7$};

\draw[fill=darkred!100] (4,2.5) -- (4.5,2.5) -- (4.5,3) -- (4,3) -- cycle;
\node at (4.25,2.75) {$8$};

\foreach \x in {0,...,6}
	\draw[color=gray!30] (0.5*\x+1.5,1)--(0.5*\x+1.5,5) ;
\foreach \x in {0,...,6}	
	\draw[color=gray!30] (1,0.5*\x+1.5)--(5,0.5*\x+1.5) ;
\draw[color=gray!30] (1,1) -- (5,1) -- (5,5) -- (1,5) -- cycle;
\draw[ultra thick] (3,1) -- (3,5);
\draw[ultra thick] (1,3) -- (5,3);

%%%%%%%%%%%%%%%%%%%%%%%%%%%%%%%%%%%%%%%%%%%%%%%%%%%%%%%%%%%%%%
%%P2P3
%%%%%%%%%%%%%%%%%%%%%%%%%%%%%%%%%%%%%%%%%%%%%%%%%%%%%%%%%%%%%%

\node [xshift=6.5cm] at (3,5.5) {Projection onto {$P_2,P_3$} subspace};

\node [xshift=6.5cm] at (3,0.25) {$P_2$};
\node [xshift=6.5cm] at (0.25,3) {$P_3$};

\foreach \i in {1,...,8}
\node [xshift=6.5cm] at (0.75+0.5*\i,0.75) {\i};

\foreach \j in {1,...,8}
\node [xshift=6.5cm] at (0.75,0.75+0.5*\j) {\j};

\draw[xshift=6.5cm, fill=blue!30] (2,1.5) -- (2.5,1.5) -- (2.5,2) -- (2,2) -- cycle;
\node [xshift=6.5cm] at (2.25,1.75) {$1$};

\draw[xshift=6.5cm, fill=pink!50] (2.5,3.5) -- (3,3.5) -- (3,4) -- (2.5,4) -- cycle;
\node [xshift=6.5cm] at (2.75,3.75) {$2$};

\draw[xshift=6.5cm, fill=green!30] (3,2) -- (3.5,2) -- (3.5,2.5) -- (3,2.5) -- cycle;
\node [xshift=6.5cm] at (3.25,2.25) {$3$};

\draw[xshift=6.5cm, fill=red!100] (4,4.5) -- (4.5,4.5) -- (4.5,5) -- (4,5) -- cycle;
\node [xshift=6.5cm] at (4.25,4.75) {$4$};

\draw[xshift=6.5cm, fill=blue!80] (1,1) -- (1.5,1) -- (1.5,1.5) -- (1,1.5) -- cycle;
\node [xshift=6.5cm] at (1.25,1.25) {$5$};

\draw[xshift=6.5cm, fill=pink!150] (1.5,4) -- (2,4) -- (2,4.5) -- (1.5,4.5) -- cycle;
\node [xshift=6.5cm] at (1.75,4.25) {$6$};

\draw[xshift=6.5cm, fill=darkgreen!100] (4.5,2.5) -- (5,2.5) -- (5,3) -- (4.5,3) -- cycle;
\node [xshift=6.5cm] at (4.75,2.75) {$7$};

\draw[xshift=6.5cm, fill=darkred!100] (3.5,3) -- (4,3) -- (4,3.5) -- (3.5,3.5) -- cycle;
\node [xshift=6.5cm] at (3.75,3.25) {$8$};

\foreach \x in {0,...,6}
	\draw[xshift=6.5cm, color=gray!30] (0.5*\x+1.5,1)--(0.5*\x+1.5,5) ;
\foreach \x in {0,...,6}	
	\draw[xshift=6.5cm, color=gray!30] (1,0.5*\x+1.5)--(5,0.5*\x+1.5) ;
\draw[xshift=6.5cm, color=gray!30] (1,1) -- (5,1) -- (5,5) -- (1,5) -- cycle;
\draw[xshift=6.5cm, ultra thick] (3,1) -- (3,5);
\draw[xshift=6.5cm, ultra thick] (1,3) -- (5,3);

\end{tikzpicture}
}
\end{tabular}
\end{center}
%
%\centering
%\begin{tabular}{cc}
%\includegraphics[trim=20mm 90mm 35mm 85mm,clip,width=0.450\textwidth]{LHS1projections.pdf} &
%\includegraphics[trim=35mm 90mm 20mm 85mm,clip,width=0.450\textwidth]{LHS2projections.pdf}\\
%%&\\
%%\includegraphics[width=0.35\textwidth]{LHS113.pdf} & \includegraphics[width=0.35\textwidth]{LHS213.pdf}\\
%%&\\
%%\includegraphics[width=0.35\textwidth]{LHS123.pdf} & \includegraphics[width=0.35\textwidth]{LHS223.pdf}\\
%\end{tabular}
\end{figure}
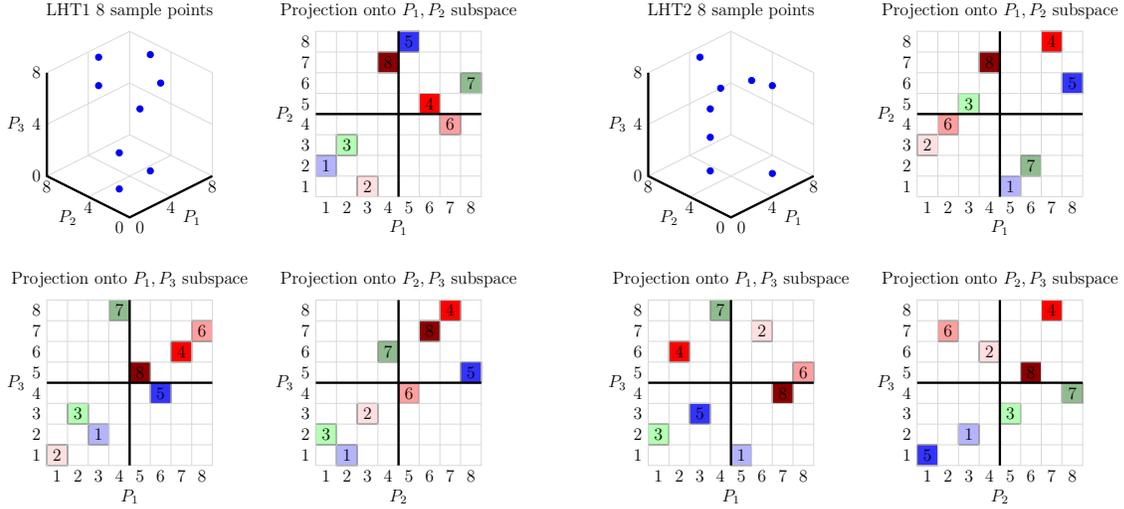

Collectively the union of LHT1 and LHT2 forms a LHS of $k=2$ trials.

While LHT1 and LHT2 are both examples of Latin Hypercube $3$-trials they exhibit different properties. The 2 dimensional subspace defined by  each of the pairs of variables 1 and 2, variables 1 and 3 and variables 2 and 3 can be partitioned into four equally sized sub-blocks as shown by the thicker lines in Figure \ref{LHS}. In LHT2 we see that the $3$-tuples (points) are evenly distributed across the four sub-blocks, while this is not the case in LHT1.

The $3$-trial LHT2 is an example of a specific space filling design known as an Orthogonal $d$-trial, where the sample points  achieve uniformity on the bivariate margins.

With this example in mind it is useful to have a formal definition for sub-blocks, Orthogonal $d$-trials and Orthogonal samples.

Let $n=p^d$ for some $p\in \Bbb N$. A $d$ dimensional parameter space (the set of all $n^d=p^{d^2}$ $d$-tuples), where each variable takes $n=p^d$  values, may be partitioned into $p^d$ {\em sub-blocks} each of which contain $p^{d^2}/p^d=p^{d(d-1)}$ points ($d$-tuples); that is,  for each $(p_1,p_2,p_3,\dots, p_d)\in [p]^d$, the set of $p^{d(d-1)}$ ordered $d$-tuples
\begin{align*}
SB_{(p_1,\dots,p_d)}&=\{((p_1,x_1),(p_2,x_2),\dots,(p_d,x_d))\mid x_i\in [p^{d-1}]\}
\end{align*}
defines a {\em sub-block}. Note that $(p_i,x_i)$ is interpreted as $(p_i-1)p^{d-1}+x_i$ and, in our examples, $p=2$.

 A Latin Hypercube $d$-trial is said to be an Orthogonal $d$-trial if the $n$ $d$-tuples are distributed evenly across all sub-blocks.
Formally, a Latin Hypercube $d$-trial  $H$ is said to be an {\em Orthogonal $d$-trial}  if $n=p^d$  and for each of the $p^d$ $d$-tuples of the form $(p_1,p_2,\dots,p_d)$, where $1\leq p_i\leq p$, there exists an element of $H$ of the form
 $((p_1,x_1),(p_2,x_2),\dots,(p_d,x_d))$, where $1\leq x_i\leq p^{d-1}$ and $(p_i,x_i)$ is interpreted as $(p_i-1)p^{d-1}+x_i$. Thus
an Orthogonal $d$-trial on $n=p^d$ values may be  represented as an $n$ by $d$ matrix where each entry is an ordered pair $(x,y)\in [p]\times [p^{d-1}]$. Further, when the matrix entries are restricted to the first coordinates,  all $p^d$ $d$-tuples on the set $[p]$ are covered and when the rows of matrix are partitioned according to the first coordinate, for each partition, the second coordinate  forms an arbitrary permutation of $[p^{d-1}]$ (that is, we have $p$ arbitrary permutations on the set $[ p^{d-1}]$).

In the above example the entries of LHT1 and LHT2 have been rewritten as ordered pairs in LHT3 and LHT4, respectively, and it is easy to see that LHT4 (LHT2) is an Orthogonal $3$-trial, while LHT3 (LHT1) is not.

\section{Results}
\label{sec:theory_results}

Here  we give  theoretical arguments that  calculate the expected coverage of the parameter space when taking the union of $k$ $d$-trials. To achieve this we begin by using combinatorial techniques to count the expected intersection sizes for a multiset of $m$ LH $d$-trials. These arguments are presented  below and then extended  to the expected coverage based on Orthogonal $d$-trials.

\subsection{The expected intersection size   of LH $d$-trials}\label{subsec:theory-lhs}

%Two LHS ($d$-trials) $H_1,H_2$  intersect in the $d$-tuple  $(a_1,a_2,\dots,a_d)\in [n]^d$ if $(a_1,a_2,\dots,a_d)\in %H_1\cap H_2.$

Because each coordinate in a LH $d$-trial contains each element of $[n]$ exactly once and a LH $d$-trial is invariant under row permutations, the number of LH $d$-trials on $[n]$ is
$n !^{\,d-1}$.

Let ${\cal M}$ be the set of all selections of $m$ LH $d$-trials (with repetition retained in each of the selections, so ${\cal M}$ is a set of multisets each of size $m$).
The number of ways to choose $q$ elements from a set of size $p$, with repetition, is ${p+q-1\choose q}={p+q-1\choose p-1},$ so
%\begin{align}\label{repallowed}
%{p+q-1\choose q}={p+q-1\choose p-1}.
%\end{align}
 \begin{align}\label{trial-repallowed}
|{\cal M}|={n!^{d-1}+m-1\choose m}.
\end{align}

\begin{theorem}\label{int-LHT} Let $M$ be a  multiset of $m$ LH $d$-trails on $[n]$; that is $M\in {\cal M}$. The expected number of ordered $d$-tuples common to all   $m$ LH $d$-trials in $M$ is given by
\begin{align}
x_m(n)=n^{d}{(n-1)!^{\,d-1}+m-1\choose m}\left/{n!^{\,d-1}+m-1\choose m}\right..\label{xmndtuple}
\end{align}
\end{theorem}

\begin{proof}

Fix a $d$-tuple $\mathbf{a}=(a_1,a_2,\dots,a_d)\in [n]^d$. There are $(n-1)!^{\,d-1}$ $d$-trials that contain this $d$-tuple. From this set the number of ways to choose, with repetition,  $m$  of these $d$-trials is
\begin{align*}
t_{\mathbf{a}}={(n-1)!^{\,d-1}+m-1\choose m}.
\end{align*}
That is, there are $t_{\mathbf{a}}$ choices of $m$ LH $d$-trials that have $\mathbf{a}$ in their intersection.
For $M\in {\cal M}$ denote the number of $d$-tuples common to all LH $d$-trials in $M$ by $c(M)$. Then
\begin{align*}
\sum_{M\in {\cal M}}c(M)=\sum_{\mathbf{a}\in [n]^d} t_{\mathbf{a}}=n^{d}{(n-1)!^{\,d-1}+m-1\choose m}.
\end{align*}
Hence the expected number ordered $d$-tuples common to all $m$ LH $d$-trials for an arbitrary $M\in {\cal M}$ is
\begin{align*}
\left(\sum_{M\in {\cal M}}c(M)\right)\frac{1}{|{\cal M}|}=n^{d}{(n-1)!^{\,d-1}+m-1\choose m}\left/{n!^{\,(d-1)}+m-1\choose m}\right..
\end{align*}

\end{proof}

\subsection{The  expected intersection size of Orthogonal $d$-trials}\label{subsec:theory-os}
For general $d$ we count the number of orthogonal $d$-trials. Thus the assumption is that $n=p^d$.

Let $H$ be an Orthogonal $d$-trial. Recall that for each $d$-tuple $(p_1,p_2,,\dots,p_d)\in [p]^d$ there is precisely one element of $H$ of the form
\begin{align*}
((p_1,x_1),(p_2,x_2),\dots,(p_d,x_d)),
 \end{align*}
 where $x_i\in [p^{d-1}]$. Further all elements of $H$ are of this form for some $(p_1,p_2,,\dots,p_d)\in [p]^d$.

 It will be useful to talk about individual coordinates in $H$ so for each $i=1,\dots, d$, let
\begin{align*}H_i(j)=&\{h\in H\mid \mbox{ the $i$-th coordinate of $h$ is $(j,x_i)$ for some }x_i\in[p^{d-1}]\}.
 \end{align*} Since $H$ is a LH $d$-trial, $|H|=p^d=n=p\cdot p^{d-1}$ and for  each $i=1,\dots, d$ and each $j=1,\dots, p$,  $|H_i(j)|=n/p=p^{d-1}$.

\begin{lemma}\label{lemA}
 The total number of Orthogonal $d$-trials on $[p]\times [p^{d-1}]$ is $(p^{d-1})!^{dp}$.
\end{lemma}

 \begin{proof} Let $H$ be an Orthogonal $d$-trial. There are $n=p^d$ $d$-tuples in $H$.  Fix $i$ and $j$, where $1\leq i\leq d$ and $1\leq j\leq p$, and define a function $f_{ij}:[p^{d-1}]\rightarrow H_i(j)$, by
 \begin{align*}
 f_{ij}(y)=((p_1,x_1),(p_2,x_2),\dots,(j,y),\dots,(p_d,x_d)),
 \end{align*}
 where $(j,y)$ is the $i$-th coordinate. Since $f_{ij}$ is a one-to-one and onto function there are $(p^{d-1})!$ different functions to choose from and $dp$ choices for $i,j$ so  $(p^{d-1})!^{\,pd}$ possible $d$-trials.\end{proof}

\begin{lemma}\label{lemB} A fixed $d$-tuple, say  $((p_1,x_1),(p_2,x_2),\dots,(p_d,x_d))$, occurs in $p^{d(d-1)(p-1)}(p^{d-1}-1)!^{dp}$  orthogonal $d$-trials on $[p]\times [p^{d-1}]$.
\end{lemma}

\begin{proof} By Lemma \ref{lemA} there are $(p^{d-1})!^{pd}$ Orthogonal $d$-trials, and each contains $n=p^d$ \mbox{$d$-tuples}. There are $n^d$ distinct $d$-tuples and any two occur the same number of times in the disjoint union of the $d$-trials. Hence a fixed $d$-tuple occurs in
 \begin{align*}
\frac{q!^{dp}\times n}{n^d}&=  (p^{d-1})!^{dp}/p^{d(d-1)}
%&=p^{d(d-1)p}((p^{d-1}-1)!^{dp}/p^{d(d-1)}\\
=p^{d(d-1)(p-1)}(p^{d-1}-1)!^{dp}
\end{align*}
Orthogonal $d$-trials.
\end{proof}

Let ${\cal M}_o$ be the set of all selections of $m$ Orthogonal $d$-trials (so ${\cal M}_o$ is a set of multisets each of size $m$).  So by Lemma \ref{lemA}
 \begin{align*}
|{\cal M}_o|={(p^{d-1})!^{dp}+m-1\choose m}={(n/p)!^{dp}+m-1\choose m}.
\end{align*}

\begin{theorem}
\label{thm:int-OT}
Let $M$ be a multiset of $m$ Orthogonal $d$-trials on $[p]\times [p^{d-1}]$ where $n=p^d$; that is $M\in {\cal M}_o$. The expected number of ordered $d$-tuples common to all $m$ Orthogonal $d$-trials    is
\begin{align*}
x_m(n)=p^{d^2}{p^{d(d-1)(p-1)}(p^{d-1}-1)!^{dp}+m-1\choose m}\left/{(p^{d-1})!^{dp}+m-1\choose m}\right..
\end{align*}

\end{theorem}

\begin{proof}
The proof follows as in the proof of Theorem \ref{int-LHT}, except that we consider ${\cal M}_o$ instead of ${\cal M}$ and the number of Orthogonal $d$-trials that intersect in a fixed $d$-tuple as established in Lemma \ref{lemB}.

%Fix a $d$-tuple $\tilde{a}=((p_1,x_1),(p_2,x_2),\dots,(p_d,x_d))$. By Lemma \ref{lemB} there are $p^{d(d-1)(p-1)}%(p^{d-1}-1)!^{dp}$  Orthogonal $d$-trials which contain this $d$-tuple. From this set the number of ways to choose, with repetition,  $m$-subsets of Orthogonal $2$-trials is
%\begin{align*}
%t_{\tilde{a}}={p^{d(d-1)(p-1)}(p^{d-1}-1)!^{dp}+m-1\choose m}.
%\end{align*}
%That is, there are $t_{\tilde{a}}$ choices of $m$ orthogonal $d$-trials that have $\tilde{a}$ in their intersection.

%For $M\in {\cal M}_o$, denote the number of $d$-tuples common to all Orthogonal $d$-trials in $M$ by $c(M)$. Then
%\begin{align*}
%\sum_{M\in {\cal M}_o} c(M)=\sum_{\tilde{a}\in [p]\times[p^{d-1}]}t_{\tilde{a}}=p^{d^2}{p^{d(d-1)(p-1)}(p^{d-1}-1)!^{dp}+m-1\choose m}.
%\end{align*}

%Hence the expected number of ordered $d$-tuples common to all $m$ Orthogonal $d$-trials for an arbitrary $M\in {\cal M}_o$ is
%\begin{align*}
%\sum_{M\in {\cal M}_o} c(M)/|{\cal M}|=p^{d^2}{p^{d(d-1)(p-1)}(p^{d-1}-1)!^{dp}+m-1\choose m}\left/{(p^{d-1})!^{dp}+m-1\choose m}\right..
%\end{align*}

\end{proof}

\subsection{The  expected size of  edgewise intersection of LH $d$-trials}
Let $1\leq i<j\leq d$. An $(i,j)$-{\em edge} of a $d$-tuple $\mathbf{a}=(a_1,a_2,\dots, a_d)$  is an ordered pair $(a_i,a_j)$. Two Latin Hypercube $d$-trials, $H_1$ and $H_2$,  are said to intersect in an $(i,j)$-edge $(a_i,a_j)$, if there exists  $(a_1,a_2,\dots,a_d)\in H_1$ and $(a_1^\prime,a_2^\prime,\dots,a_d^\prime)\in H_2$, such that $a_i=a_i^\prime$ and $a_j=a_j^\prime$.

There are ${d\choose 2}$ edges in a $d$-tuple, so the total number of possible distinct edges is $n^2{d\choose 2}$. In addition, there are $n$ $d$-tuples in a LH $d$-trial, so there are
 $n{d\choose 2}$ edges in total in a $d$-trial.

 \begin{lemma}\label{lemC}
 A fixed $(i,j)$-edge $(a_i,a_j)$ is contained in $(n-1)!n!^{d-2}$ distinct LH $d$-trials.
 \end{lemma}

 \begin{proof}
Multiplying the number of distinct LH $d$-trials by the number of edges in a LH $d$-trial and dividing by the total number of distinct edges  counts the number of LH $d$-trials that contain a fixed $(i,j)$-edge; that is,

\begin{align*}
\frac{n!^{\,d-1}\times n{d\choose 2}}{n^2{d\choose 2}}=(n-1)! n!^{\,d-2}.
\end{align*}
\end{proof}

We now count the expected  number of edges common to all LH $d$-trials from a selection $M\in {\cal M}$.

\begin{theorem}\label{th:x2n} Let $M$ represent a multiset of $m$ LH $d$-trials on $[n]$; that is, $M\in {\cal M}.$ Then the expected number of edges common to all $m$ LH $d$-trials in $M$ is
\begin{align}
x_{m}(n)=n^t{d\choose t}{(n-1)!^{\,d-1}n^{d-2}+m-1\choose m}\left/{n!^{\,d-1}+m-1\choose m}, \quad t=2\right.;\label{x2n}
\end{align}
that is, $x_m(n)$ is the expected intersection in the projection to a subspace of dimension $t=2$\end{theorem}

\begin{proof}
The case $d=2$ is covered in Theorem \ref{int-LHT}.  For general $d$ we fix an $(i,j)$-edge, say $(a_i,a_j)$. By Lemma \ref{lemC} there are  $(n-1)!n!^{d-2}$ LH $d$-trials that contain this edge. From this set the number of ways to choose, with repetition,  $m$ of these LH  $d$-trials is
\begin{align*}
s_{(a_i,a_j)}={(n-1)!(n!)^{d-2}+m-1\choose m}.
\end{align*}
That is, there are $s_{(a_i,a_j)}$ choices of $m$ LH $d$-trials that intersect in the fixed $(i,j)$-edge $(a_i,a_j)$.

For $M\in {\cal M}$ denote the number of edges common to all LH $d$-trials by $c(M)$. Then, denoting the set of $(i,j)$-edges by $E$,
\begin{align*}
\sum_{M\in {\cal M}} c(M)=\sum_{1\leq i<j\leq d}\left(\sum_{(a_i,a_j)\in E}s_{(a_i,a_j)} \right)=n^2{d\choose 2}{(n-1)!(n!)^{d-2}+m-1\choose m}.
\end{align*}
Hence the expected number of edges common to all $m$ LH $d$-trials for an arbitrary $M\in {\cal M}$  is
\begin{align*}
\left(\sum_{M\in {\cal M}} c(M)\right)\frac{1}{|{\cal M}|}=n^2{d\choose 2}{(n-1)!n!^{d-2}+m-1\choose m} \left/ {(n!)^{d-1}+m-1\choose m}\right..
\end{align*}
The result follows by noting that $(n-1)!n!^{d-2}= (n-1)!^{d-1}n^{d-2}.$
\end{proof}

In \cite{BBDT} the authors used MATLAB simulations to test the percentage coverage of $t$ dimensional subspaces at the sub-block level, where we recall that
 for a $d$ dimensional parameter space with $n=p^d$,  the set of $p^{d(d-1)}$ ordered $d$-tuples
\begin{align*}
SB_{(p_1,\dots,p_d)}&=\{((p_1,x_1),(p_2,x_2),\dots,(p_d,x_d))\mid x_i\in [p^{d-1}]\}
\end{align*}
defines a  sub-block. These simulations were focused on testing the percentage coverage for $(i,j)$-edges from the set
\begin{align*}
E_{(p_i,p_j)}=\{((p_i,x_i),(p_j,x_j))\mid x_i,x_j\in [p^{d-1}]\},\end{align*}
where $i,j\in [d]$ and $p_i,p_j\in [p]$ are fixed. Note $| E_{(p_i,p_j)}|=(p^{d-1})^2$. Theorem \ref{th:x2n} allows us to calculate the expected coverage of this set of $(i,j)$-edges.

\begin{corollary}
\label{thm:int-edges}
Let $n=p^d$. Further let $M$ be a multiset of $m$ LH $d$-trials on $[p^d]$; that is $M\in {\cal M}$. Fix $i,j\in [d]$ and $p_i,p_j\in [p]$. Then the expected number of $(i,j)$-edges in $E_{(p_i,p_j)}$ and  common to all   $m$ LH $d$-trials  in $M$ is
\begin{align}
x_m(n)=p^{2d-2}{(p^d-1)!^{\,d-1}p^{d^2-2d}+m-1\choose m}\left/{p^d!^{\,d-1}+m-1\choose m}\right..\label{x2pd}
\end{align}
\end{corollary}

\begin{proof}
Theorem \ref{th:x2n} gives the expected number of edges common to all $m$ LH $d$-trials in $M$ (that is, $i$ and $j$ are not fixed). However we are interested in the expected number of edges in $E_{(p_i,p_j)}$ (with $i$ and $j$ fixed).
There are $\binom{d}{2}$ choices for the pair $(i,j)$, where $1 \leq i < j \leq d$. Also rather than summing over all $n^2=p^{2d}$ $(i,j)$-edges we sum over the $p^{2d-2}$ edges in $E_{(p_i,p_j)}$.
Therefore to evaluate $x_m(n)$ we  divide the result from Theorem \ref{th:x2n} by $p^2\binom{d}{2}$.
\end{proof}

%\begin{proof}
 %Theorem \ref{th:x2n} gives the expected number of ordered edges common to all   $m$ LH $d$-trials in $M$ (that is, $i$ and $j$ are not fixed). However we are  interested in the expected number of edges in $E_{(p_i,p_j)}$ (with $i$ and $j$ fixed). So repeating the argument gives:

%For a fixed $(i,j)$-edge, say $((p_i,x_i),(p_j,x_j))$, there are $(n-1)!n!^{d-2}$ LH $d$-trials that contain this edge. From this set the number of ways to choose, with repetition, $m$ of these LH $d$-trials is
%\begin{align*}
%s_e^\prime={(n-1)!(n!)^{d-2}+m-1\choose m}.
%\end{align*}
%That is, there are $s_e^\prime$ choices of $m$ LH $d$-trials that have the fixed $(i,j)$-edge $e$ in common.

%Summing over all edges in $E_{(p_i,p_j)}$ gives
%\begin{align*}
%\sum_{e\in E_{(p_i,p_j)}}s_e^\prime= (p^{d-1})^2{(p^d-1)!(p^d!)^{d-2}+m-1\choose m}.
%\end{align*}
%So the expected number of $(i,j)$-edges in $m$ multisets $E_{(p_i,p_j)}$ have in common to all $m$ LH $d$-trials in an arbitrary $M\in {\cal M}$ is
%is obtained by dividing by the total number of  $m$ sets of $d$-trials, with repetition allowed giving
%\begin{align*}
%\left(\sum_{e\in E_{(p_i,p_j)}}s_e^\prime\right) \frac{1}{|{\cal M}|}=p^{2d-2}{(p^d-1)!(p^d!)^{d-2}+m-1\choose m} \left/ {(p^d!)^{d-1}+m-1\choose m}\right..
%\end{align*}
%Noting $(p^d-1)!(p^d)!^{d-2}= (p^d-1)!^{d-1}p^{d^2-2d}$ yields the result.
%Or with $n=p^d$
%\begin{align*}
%x_{m}(n)=(\frac{n}{p})^2{(n-1)!^{d-1}!n^{d-2}+m-1\choose m} \left/ {(n!)^{d-1}+m-1\choose m}\right..
%\end{align*}

%\end{proof}

\noindent {\bf Remark:}
%Theorem 4 applies for a particular block of dimension $p^{d-1}$.  Since there are $p^2$ blocks then we just have to multiply (5) by $p^2$ to get the expected coverage over the whole space.  Furthermore, t
There is a natural extension of this result to projection onto a subspace of arbitrary dimension $t > 2.$

\subsection{Bounds on percentage coverage of $d$-tuples}\label{subsec:theory-bds}

To estimate the number, $U(k,n)$, of cells in the parameter space covered by the union of $k$ $d$-trials, with $n$ partitions for each of the $d$ parameters, we count via the Principle of Inclusion/Exclusion obtaining
\begin{equation}
U(k,n) = \sum_{m=1}^k (-1)^{m+1} {k\choose m} x_m(n),
\label{eq:eq1}
\end{equation}
where $x_m(n)$ denotes the expected intersection size of $m$ arbitrary  trials depending on the sampling strategy. We define the expected coverage fraction to be
$$P(k,n) = \frac{U(k,n)}{n^d}.$$

 We have from Theorems \ref{int-LHT}, \ref{thm:int-OT} and \ref{th:x2n} three different expressions for the $x_m(n)$.  So let the expected numbers of ordered $d$-tuples in the case of LHS, OS and sub-block coverage for $t=2$ be, respectively,
 $x_{mL}(n)$, $x_{mO}(n)$ and $x_{m2}(n)$ then from Theorems \ref{int-LHT}, \ref{thm:int-OT} and \ref{th:x2n} we have

 \begin{align*}
x_{mL}(n)&=n^{d}\prod_{i=0}^{m-1}\frac{a+i}{b+i}, \quad a=(n-1)!^{d-1}, \quad b=n!^{d-1}\\
x_{mO}(n)&=n^{d}\prod_{i=0}^{m-1}\frac{a+i}{b+i}, \quad a=p^{d(d-1)(p-1)}(p^{d-1}-1)!^{dp}, \quad b=(p^{d-1})!^{dp}; \mbox{ and }\\
x_{m2}(n)&=n^{2}\prod_{i=0}^{m-1}\frac{a+i}{b+i}, \quad a=(n-1)!^{d-1}n^{d-2}, \quad b=n!^{d-1}.
\end{align*}
Note that in the case that $d=2$ then $x_{mL}(n) = x_{m2}(n).$

Now the binomial expansion  gives
\begin{align}
\sum_{m=0}^k{k\choose m}u^m=(1+u)^k.\label{eq:binom}
\end{align}
Also it is easy to see that, for $x\geq 0$,
\begin{align}
1+x\leq e^x&\mbox{ and } 1-e^{-x}\leq x,\label{eq:exponential1}
\end{align}
while, for $0\leq t<1$,
\begin{align}
e^{\frac{t}{2}}-1 \leq t&\mbox{ and }-\frac{t^2}{2}\leq t+\ln (1-t)\leq \frac{-t^2}{4}.
\label{eq:exponential2}
\end{align}
Moreover, for $0<a\leq b$ and $i\geq 0$,
\begin{align}
\frac{a}{b}\leq\frac{a+i}{b+i}\leq\frac{a}{b}\left(1+\frac{i}{a}\right)\leq \frac{a}{b}\exp\left({\frac{i}{a}}\right).\label{ineq:1}
\end{align}
Thus for $0\leq i\leq m-1\leq k-1$
\begin{align}
\left(\frac{a}{b}\right)^m\leq\prod_{i=0}^{m-1}\left(\frac{a+i}{b+i}\right)\leq\left(\frac{a}{b}\right)^m
\prod_{i=0}^{m-1}\left(1+\frac{i}{a}\right)\leq \left(\frac{a}{b}\right)^m\exp\left({\frac{k(k-1)}{2a}}\right), \label{eq:prod}
\end{align}
and for $0\leq t=\frac{k(k-1)}{a}\leq 1,$ using \eqref{eq:exponential2}
\begin{align}
0\leq \prod_{i=0}^{m-1}\left(\frac{a+i}{b+i}\right)-\left(\frac{a}{b}\right)^{m}\leq \left(\frac{a}{b}\right)^{m}\left(
\exp\left(\frac{k(k-1)}{2a}\right)-1\right)\leq \left(\frac{a}{b}\right)^{m}
\frac{k(k-1)}{a}.\label{bound}
\end{align}
We relate this back to the expression for the $x_m(n)$ for a general $a$ and $b$ with $0<a\leq b$ and
\begin{align*}
x_m(n)= n^{d}\prod_{i=0}^{m-1}\frac{a+i}{b+i}.
\end{align*}

Recalling that $P(k,n)$ denotes the expected coverage fraction of  $d$-tuples in the parameter space by taking the union of $k$  trials with either LHS or OS, we have the following result.

\begin{theorem}
\label{thm:LHS-OS-form}
In the case of Latin Hypercube sampling and Orthogonal sampling (with $n=p^d$)
\begin{equation}
P(k,n)\sim \left(1-\exp(-k \lambda)\right)\mbox{ as }k \lambda^2\to 0, \quad \lambda=\frac {1}{n^{d-1}}.
\end{equation}
\end{theorem}

\begin{proof}
We begin by using the Principle of Inclusion/Exclusion, using \eqref{eq:binom}  and evaluating $P(k,n)$ in terms of the general form of $x_m(n)$, as follows
\begin{align*}
P(k,n)&=\sum_{m=1}^k(-1)^{m+1}{k\choose m}x_m(n)/n^d=\sum_{m=1}^k(-1)^{m+1}{k\choose m}\prod_{i=0}^{m-1}\frac{a+i}{b+i}\\
&=1-\sum_{m=0}^k(-1)^{m}{k\choose m}\lambda^m+E_1 =1-\exp(-k \lambda)+E_2+E_1,
\end{align*}
where $\lambda=\frac{a}{b}$,
\begin{align*}
E_1&= \sum_{m=1}^{k}(-1)^m{k\choose m} \left[\lambda^m-\prod_{i=0}^{m-1}\left(\frac{a+i}{b+i}\right)\right]
\end{align*}
and
\begin{align*}
E_2&= \exp (-k\lambda)-(1-\lambda)^k.
\end{align*}
It follows from \eqref{bound} and then \eqref{eq:binom} and  \eqref{eq:exponential1} that
\begin{align*}
|E_1|&\leq \sum_{m=0}^{k}{k\choose m} \lambda^m\frac{k(k-1)}{a}\leq \exp(k \lambda)\frac{k(k-1)}{a}.
\end{align*}
Moreover, it follows from \eqref{eq:exponential1} and \eqref{eq:exponential2} that
\begin{align*}
|E_2|&= \exp (-k\lambda)\left|1-\exp(k\lambda+k\ln\left(1-\lambda \right))\right|\leq \exp (-k\lambda)k\lambda^2.
\end{align*}
As $n\rightarrow \infty$
\begin{align*}
E=E_1+E_2
&={\rm O}\left(1-\exp(-k\lambda)\right)\mbox{ and}\\
E=E_1+E_2&={\rm O}\left(\exp(-k\lambda)\right),
\end{align*}
provided $k\lambda\leq Cn$; note it follows from Stirling's formula  that $\frac{e^{2k\lambda}k^2}{a}\rightarrow 0$ in this case.
Thus
\begin{align*}
P(k,n)\sim \left(1-\exp(-k\lambda)\right),\mbox{ as }k\lambda^2\to 0.
\end{align*}

Finally, in the case of LHS with $a=(n-1)!^{d-1}$ and $b=n!^{d-1}$ then
\begin{align*}
\lambda=\frac{1}{n^{d-1}};
\end{align*}
while with OS $ a=(p^{d-1}-1)!^{dp}p^{d(d-1)(p-1)}$ and $b=(p^{d-1})!^{dp}$ and so

\begin{align*}
\lambda&=p^{d(d-1)(p-1)}\left(\frac{(p^{d-1}-1)!}{(p^{d-1})!}\right)^{dp}=p^{d(d-1)(p-1)-(d-1)dp}=\frac{1}{p^{d(d-1)}}=\frac{1}{n^{d-1}}.
\end{align*}

Thus $\lambda$ is the same in both cases and the percentage coverage is the same in both cases (assuming that $n=p^d$ for the OS case) and so the result is proved.
\end{proof}

We can extend this analysis to the case of the  2 dimensional sub-block projection ($t=2$) but now
$P(k,n) = \frac{U(k,n)}{n^2}.$

\begin{theorem}
\label{thm:sub-block-form}
For the 2 dimensional sub-block projection
\begin{equation}
P(k,n)\sim \left(1-\exp(-k \lambda)\right)\mbox{ as }k \lambda^2\to 0, \quad \lambda=\frac {1}{n}.
\end{equation}
\end{theorem}

\begin{proof}
With $ a=(n-1)!^{d-1}n^{d-2} $ and $ b=n!^{d-1}$ then

\begin{align*}
\lambda&=\frac{(n-1)!^{d-1}n^{d-2}}{n!^{d-1}}=n^{d-2}\frac{(n-1)!^{d-1}}{n!^{d-1}}=\frac{1}{n}.
\end{align*}
\end{proof}

We now present some simulation results confirming our  theoretical results.  We focus on the case $d=5$ and give expected coverage at the 25, 50, 75 and 100 percent levels for $t=2, 3, 4$ by plotting the logarithm of the number of trials as a function of the logarithm of $n.$

\section{Discussion and  conclusions}

Theorem \ref{thm:LHS-OS-form} states that the expected coverage of both LHS and OS is of the form $1-exp(-k\lambda)$ where $\lambda =\frac{1}{n^{d-1}}$, while Theorem \ref{thm:sub-block-form} states that the expected coverage when projecting onto a $t$ dimensional subspace with $t=2$ has the same form but now $\lambda=\frac{1}{n}$ and this coverage is independent of $d$.  Although, we have not presented the analysis here we can extend the results of Theorem \ref{thm:sub-block-form} to arbitrary $t$ so that $\lambda =\frac{1}{n^{t-1}}$.  Figure \ref{d5fig} confirms these results.  In all but the full coverage case the gradient of the straight line is $t-1.$  In the case of the full coverage the gradient appears to behave as $t-1/2.$ We think this is partly due to the effect  that
as the percentage coverage increases then the higher is the rate of overlapping $d$-trials.  Nevertheless our numerical results are consistent with the theory.

In conclusion, we have obtained analytical  results for the expected coverage of parameter space when using both Latin Hypercube and Orthogonal sampling.  We have shown that there is no difference between the  two in terms of the expected coverage.  We have  obtained analytical results of the expected  coverage when projecting onto  small dimension subspaces.  In this case the expected coverage is independent of the dimension of  the parameter space and depends only on the dimension of the projected subspace.  These result have relevance in the construction of Experimental Designs.  The analytical results are supported by simulations.

\section {Acknowledgements}
Authors 1 and 5 wish to thank The University of Queensland's Centre for Coal Seam Gas (CCSG) for their support. 
Author 2 wishes to thank Blanca Rodriguez, Alfonso Bueno-Orovio and Ollie Britton and the Computational Cardiovascular group in the Department of Computer Science at the University of Oxford for on-going and deep discussions about Populations of models.
 Author 6 wishes to acknowledge the support from TUBITAK 2219.

\begin{figure}[htb]
\caption{Coverage of a $d=5$-dimensional parameter space when projected onto $t=2, 3, 4$-dimensional subspaces} \label{d5fig}
\centering
$$\begin{array}{cc}
\includegraphics[width=0.45\textwidth]{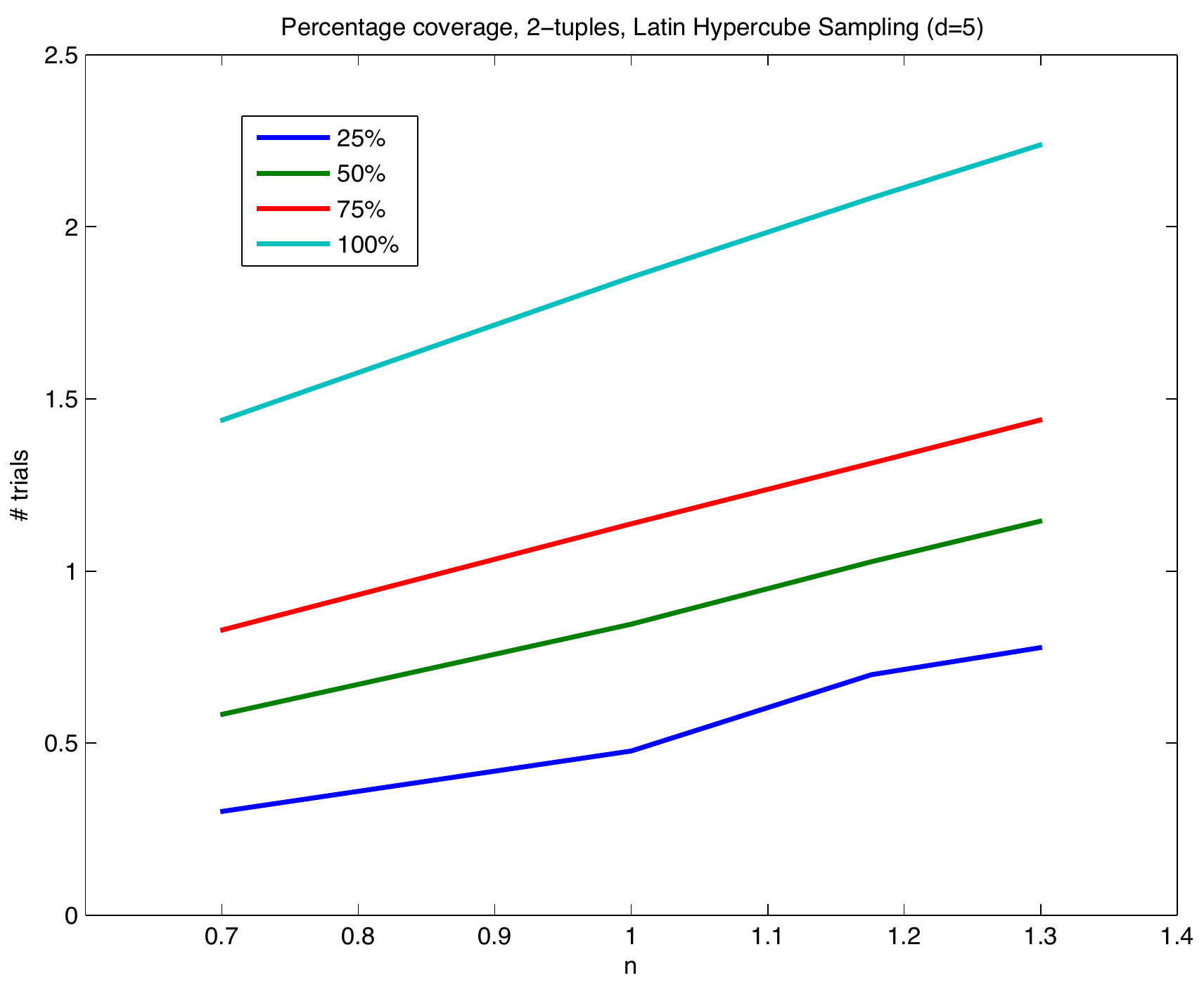} & \includegraphics[width=0.45\textwidth]{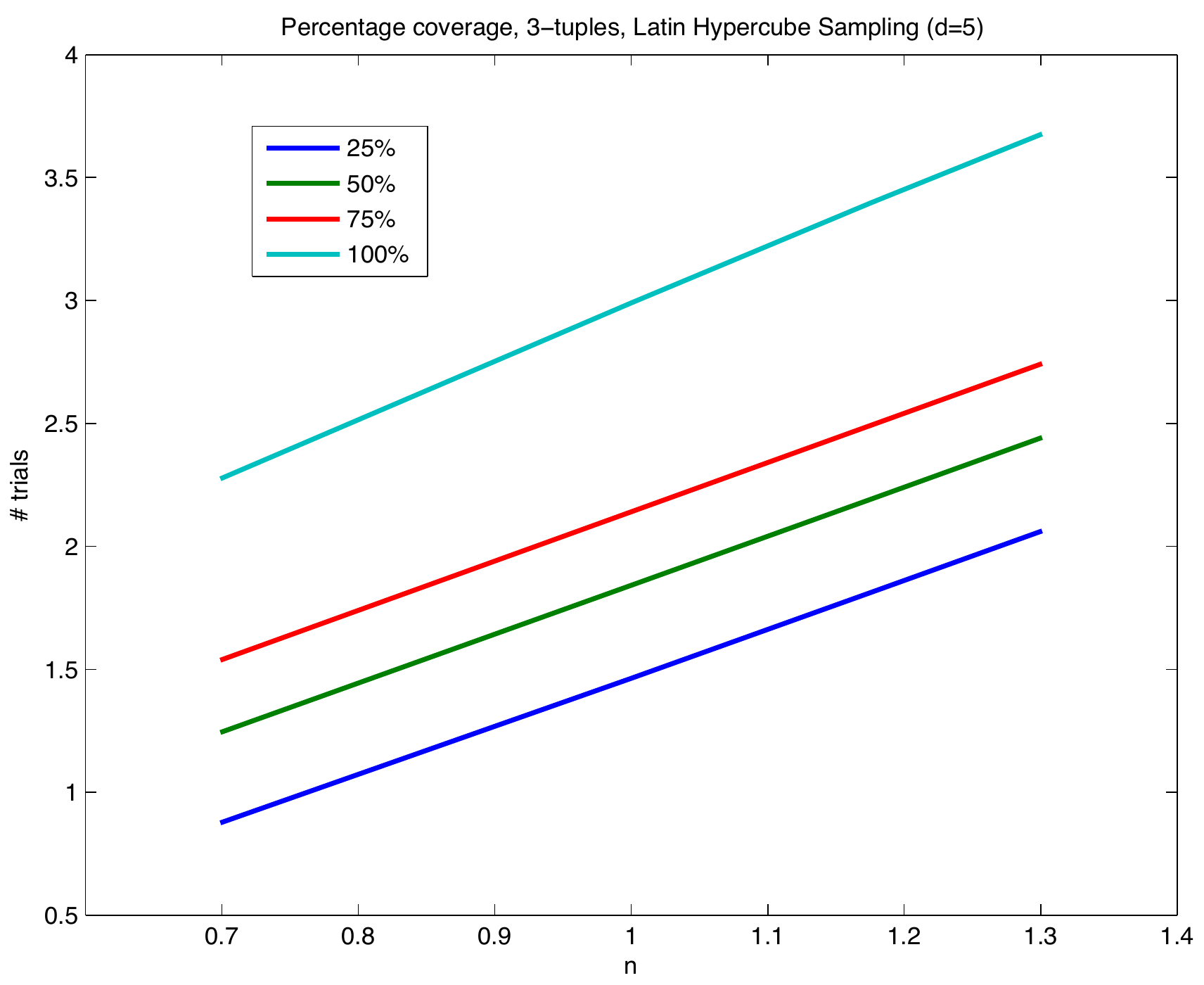}
\end{array}$$
$$\begin{array}{c}
\includegraphics[width=0.45\textwidth]{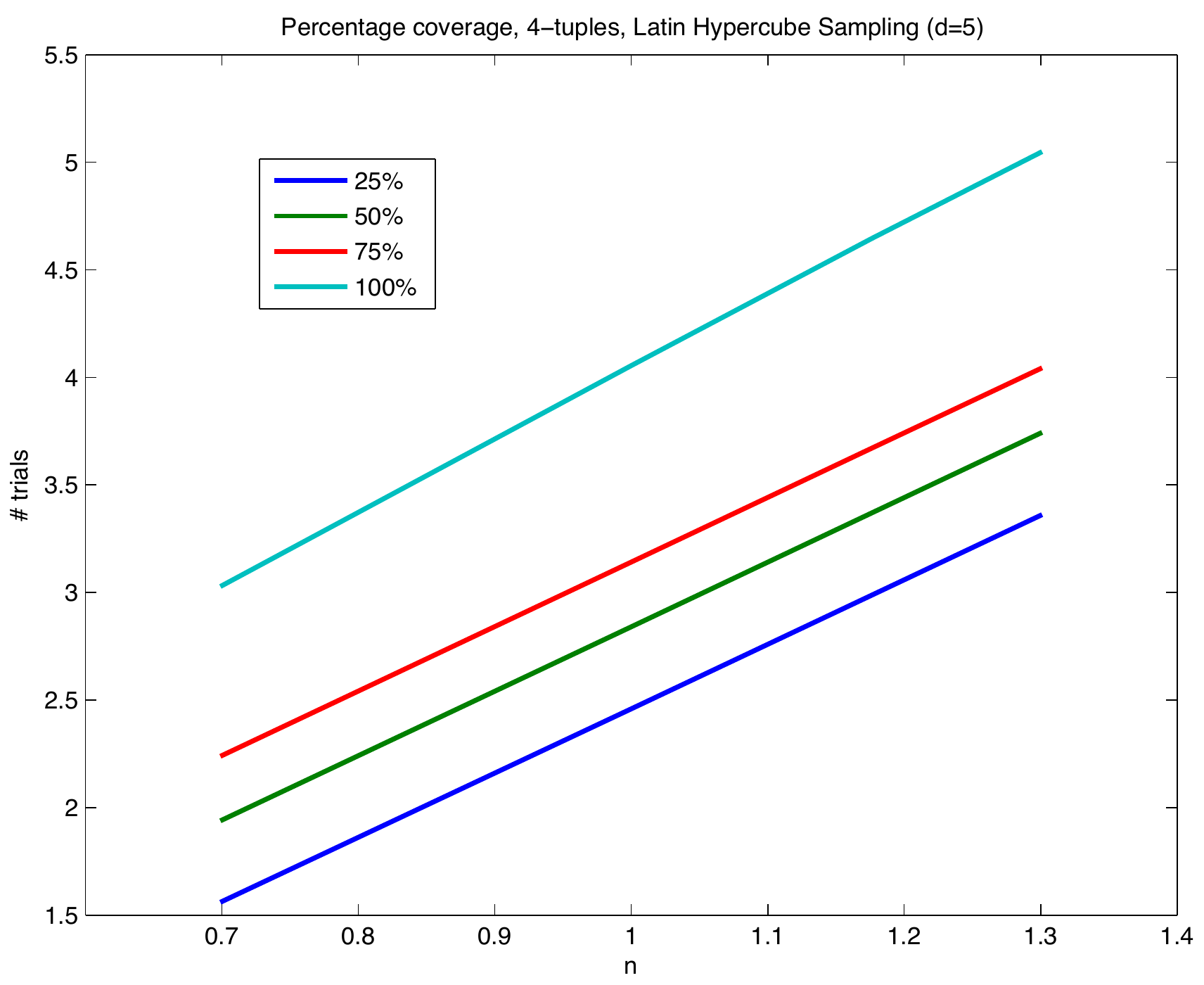}
\end{array}$$
%\begin{tabular}{c}
%\includegraphics[trim=20mm 90mm 35mm 85mm,clip,width=0.450\textwidth]{d5t2.pdf} \\
%\end{tabular}
\end{figure}

 \bibliographystyle{plain}
    \bibliography{lhs-ref}

\begin{thebibliography}{10}

\bibitem{A1}
O.J. Britton, A.~Bueno-Orovio, K.~Van Ammel, H.R. Lu, R.~Towart, D.J.
  Gallacher, and B.~Rodriguez.
\newblock Experimentally calibrated population of models predicts and explains
  inter subject variability in cardiac cellular electrophysiology.
\newblock {\em PNAS}, 110(23), Jun 2013.

\bibitem{BBDT}
K.~Burrage, P.~Burrage, D.~Donovan, and H.B. Thompson.
\newblock Populations of models, experimental designs and coverage of parameter
  space by {L}atin hypercube and orthogonal sampling.
\newblock {\em Procedia Computer Science}, 51:1762--1771, 2015.

\bibitem{A4}
K.~Burrage, P.M. Burrage, D.~Donovan, T.A. McCourt, and H.B. Thompson.
\newblock Estimates on the coverage of parameter space using populations of
  models.
\newblock {\em Modelling and Simulation, IASTED, ACTA Press}, pages DOI:
  10.2316/P.2014.813--013, 2014.

\bibitem{A3}
A.~Carusi, K.~Burrage, and B.~Rodriguez.
\newblock Bridging experiments, models and simulations: {A}n integrative
  approach to validation in computational {C}ardiac {E}lectrophysiology.
\newblock {\em Am. J. Physiology}, 303(2):H144--55, 2012.

\bibitem{B2}
E.T.Y. Chang, M.~Strong, and R.H. Clayton.
\newblock Bayesian sensitivity analysis of a cardiac cell model using a
  gaussian process emulator.
\newblock {\em PLoS ONE}, 10(6), June 2015.

\bibitem{CG1}
S-K. Choi, R.V. Gandhi, R.A. Canfield, and C.L. Pettit.
\newblock Polynomial chaos expansion with {L}atin {H}ypercube sampling for
  estimating response variability.
\newblock {\em AIAA Journal}, 42(6):1191--1198, 2004.

\bibitem{CD}
C.J. Colbourn and (Eds) J.H.~Dinitz.
\newblock {\em Handbook of {C}ombinatorial {D}esigns, {S}econd {E}dition}.
\newblock Chapman \& Hall/CRC, Boca Raton, FL, 2006.

\bibitem{A5}
C.C. Drovandi, A.N. Pettitt, and M.J. Faddy.
\newblock Approximate {B}ayesian computation using indirect inference.
\newblock {\em Journal of the Royal Statistical Society: Series C (Applied
  Statistics)}, 60(3):317--337, 2011.

\bibitem{B3}
A.M.S. Dutta, J.~Walmsley, and B.~Rodriguez.
\newblock Ionic mechanisms of variability in electrophysiological properties.
\newblock {\em Ischemia}, 40:691--694, 2013.

\bibitem{B4}
P.~Gemmell, K.~Burrage, B.~Rodriguez, and T.A. Quinn.
\newblock Population of computational rabbit-specific ventricular action
  potential models for investigating sources of variability in cellular
  repolarisation.
\newblock {\em PLoS ONE}, 9(2), 2014.

\bibitem{LSZ}
H.~Li, P.~Sarma, and D.~Zhang.
\newblock A comparative study of the probabilistic-collocation and
  experimental-design methods for petroleum-reservoir uncertainty
  quantification.
\newblock {\em SPE Journal}, 16(2):429--439, 2011.

\bibitem{MT}
E.~Marder and A.L. Taylor.
\newblock Multiple models to capture the variability of biological neurons and
  networks.
\newblock {\em Computation and Systems, Nature Neuroscience}, 14(2):133--138,
  2011.

\bibitem{MBC}
M.D. McKay, R.J. Beckman, and W.J. Conover.
\newblock A comparison of three methods for selecting values of input variables
  in the analysis of output from a computer code.
\newblock {\em Technometrics}, 21(2):239--245, 1979.

\bibitem{B1}
C.~Sanchez, A.~Bueno-Orovio, E.~Wettwer, S.~Loose, J.~Simon, U.~Ravens,
  E.~Pueyo, and B~Rodriguez.
\newblock Inter-subject variability in human atrial action potential in sinus
  rhythm versus chronic atrial fibrillation.
\newblock {\em PLoS ONE}, 9(8), August 2014.

\bibitem{stein}
M.~Stein.
\newblock Large sample properties of simulations using {L}atin {H}ypercube
  sampling.
\newblock {\em Technometrics}, 29(2):143--151, 1987.

\bibitem{A7}
B.~Tang.
\newblock Orthogonal array-based {L}atin hypercubes.
\newblock {\em Journal of the American Statistical Association},
  88(424):1392--1397, 1993.

\bibitem{A2}
J.~Walmsley, J.F. Rodriguez, G.R. Mirams, K.~Burrage, I.R. Efimov, and
  B.~Rodriguez.
\newblock {MRNA} expression levels in failing human hearts predict cellular
  electrophysiological remodelling: {A} population based simulation study.
\newblock {\em PLoS ONE}, 8(2), 2013.

\bibitem{A8}
W.J. Welch, R.J. Buck, J.~Sacks, H.P. Wynn, T.J. Mitchell, and M.D. Morris.
\newblock Screening, predicting, and computer experiments.
\newblock {\em Technometrics}, 34:15--25, 1992.

\end{thebibliography}

\end{document}